\documentclass[12pt,draft]{amsart}
\usepackage{amsfonts}
\usepackage{amssymb}

\usepackage[cp1251]{inputenc}
\usepackage[russian,english]{babel}

\newtheorem{teo}{Theorem}[section]
\newtheorem{lem}[teo]{Lemma}
\newtheorem{cor}[teo]{Corollary}
\newtheorem{dfn}[teo]{Definition}
\newtheorem{rk}[teo]{Remark}

\newtheorem{ex}[teo]{Example}

\def\<{\langle}
\def\>{\rangle}

\def\be{\beta}
\def\L{\Lambda}
\def\l{\lambda}
\def\s{\sigma}

\def\a{\alpha}
\def\om{\omega}
\def\e{\varepsilon}
\def\f{{\varphi}}
\def\F{{\Phi}}
\def\A{{\mathcal A}}
\def\M{{\mathcal M}}
\def\cN{{\mathcal N}}

\def\KI{{\mathcal K_{\text{I}}}}
\def\KII{{\mathcal K_{\text{II}}}}
\def\KIII{{\mathcal K_{\text{III}}}}
\def\KIV{{\mathcal K_{\text{IV}}}}

 \def\N{{\mathbb N}}
 
\begin{document}

\selectlanguage{english}

\title[Topological and frame properties of certain pathological $C^*$-algebras
]
{Topological and frame properties of certain pathological $C^*$-algebras
}

\author{D.V. Fufaev}

\thanks{The work was supported by the Foundation for the
Advancement of Theoretical Physics and Mathematics ``BASIS''}

\address{Moscow Center for Fundamental and Applied Mathematics,
\newline
Dept. of Mech. and Math., Lomonosov Moscow State University}

\email{denis.fufaev@math.msu.ru, fufaevdv@rambler.ru}

\begin{abstract}
We introduce a classification of locally compact Hausdorff topological spaces with respect to the behavior of $\s$-compact subsets, and relying on this classification
%from the point of 
we study properties of corresponding $C^*$-algebras %from the
in terms of frame theory and the theory of $\A$-compact operators in Hilbert $C^*$-modules, some pathological examples are constructed.
\end{abstract}

\maketitle

\section*{Introduction}

Originally the concept of frames was introduced for the Hilbert spaces theory, and then by Frank and Larson in 
\cite{FrankLarson1999} and \cite{FrankLarson2002} was generalized for the case of Hilbert $C^*$-modules.
During the last decade it has been developing intensively
(see, e.g., \cite{ArBak2017}, \cite{Bak2019}).
It turns out that there is a connection between frame theory and the theory of $\A$-compact operators and uniform structures.

It is well-known that bounded linear operator in Hilbert space is compact (i.e. can be approximated in norm by operators of finite rank) if and only if it maps the unit ball to a totally bounded set with respect to the norm. For the case of Hilbert $C^*$-modules, i.e. if we consider some $C^*$-algebra instead of the scalar field $\mathbb C$ (in this case operators are called $\A$-compact) it is not true: indeed, even in the case of any infinite-dimensional unital $C^*$-alebra $\A$ the identity operator has finite rank (which is equal to one), but the unit ball is not totally bounded. So a question how to describe the $\A$-compactness in geometric terms arose.

First steps were done in papers \cite{KeckicLazovic2018}, \cite{lazovic2018}.
In \cite{Troitsky2020JMAA} by E.V. Troitsky a significant development was obtained, namely, a specific uniform structure was constructed and it was proved that if $\cN$ is a countably generated Hilbert $C^*$-module then an adjointable operator
$F:\M\to\cN$ is $\A$-compact if and only if the image of the unit ball of $\M$ is totally bounded in $\cN$ with respect to this uniform structure.

In \cite{TroitFuf2020} it was proved that $\A$-compactness of an operator implies totally boundedness of the image of the unit ball with respect to this uniform structure for any Hilbert $C^*$-module $\cN$. The inverse statemant was proved for modules $\cN$
which could be represented as an orthogonal direct summand in the standard module over the unitalization algebra $\dot{\A}$ (which is equal to $\A$ itself in unital case) for some cardinality --- that is, in the module of the form $\bigoplus_{\l\in\Lambda}\dot{\A}$. 
In particular it holds for modules which could be represented as an orthogonal direct summand in $\bigoplus_{\a\in\Lambda}{\A}$ in case when $C^*$-algebra $\A$ is countably generated as a module over itself (this in fact is equivalent to the fact that $\A$ is $\s$-unital, i.e. it has a countable approximate unit, see. \cite[Proposition 2.3]{Asadi2016}).

But in the case when $\A$ is not $\s$-unital in \cite{Fuf2021faa} a counterexample was constructed, namely, an algebra $\A$, considered as a module over itself, such that the identity operator from $\A$ into itself is not $\A$ -compact, but the unit ball is totally bounded with respect to the introduced uniform structure (and even with respect to some stronger one).
The constructed algebra is commutative, so this result
makes the study of underlying topological space
interesting.

Also it turns out that the existence of a representation of a module $\cN$ as an orthogonal direct summand in $\bigoplus_{\l\in\Lambda}\dot{\A}$ is equivalent to the existence of a standard frame in $\cN$, so there is a connection between stated problem on $\A$-compact operators and the frame theory in Hilbert $C^*$-modules, and we will study constructed $C^*$-algebras from this point of view too.

More precisely, the existence of a standard frame is sufficient to satisfy the $\A$ -compactness criterion. Hence, by proving that the criterion is not satisfied for some module, we, as a consequence, show that there is no standard frame for this module
(which, by the way, does not mean that there is no outer standard frame in sense of \cite{ArBak2017}, see \cite[Remark 3.3]{Fuf2021faa}).
In \cite{HLi2010}, \cite{Asadi2016}, \cite{AmAs2016},\cite{FrAs2020} there is also a searching for modules without frames, more precisely, for algebras $\A$ such that it is possible to say surely is there an $\A$-module without frames or not.

%We introduce several classes of locally compact Hausdorff topological spaces decreasing on a scale with respect to some property.

We introduce a scale of classes of locally compact Hausdorff topological spaces that decreases with respect to some property.

$\KI$ --- $\s$-compact spaces (i.e. spaces which could be covered by countable family of compact subsets);

$\KII$ --- non-$\s$-compact spaces which have a dense $\s$-compact subset;

$\KIII$ --- spaces in which no $\s$-compact subset is dense, that is, the complement to any $\s$-compact subset has an inner point, but the point at infinity (in a one-point compactification) may not be inner for the complement; equivalently, there exists a $\s$-compact not precompact subset;

$\KIV$ --- spaces in which the complement to any $\s$-compact subset has an inner point, and (in a one-point compactification) the point at infinity is always inner for the complement; equivalently, every $\s$-compact subset is precompact.

In \S1 some preliminaries on Hilbert $C^*$-modules, frames and topological spaces are given and some properties of $\KI$ and $\s$-unital algebras are established.

In \S2 we obtain some properties of algebras $C_0(K)$ for $K\in\KII$: we prove that such algebras never have standard frames (theorem \ref{ii_nostfr}).
Nevertheless, this does not imply any conclusion about the existence of a non-standard frame (which we will define later): it is proved that in the case of separable space it really does not exist (theorem \ref{ii_nofr}), but also an example of a space when such a frame exists is given (example \ref{ii_exfr}).
Also an example when the $\A$-compactness criterion does not hold is given, since the absence of a standard frame is not enough for this; there is also no non-standard frames for this example (example \ref{ii_exnocr}, theorems \ref{ii_nocr} and \ref{ii_nocr_nofr}).

In \S3 we study the algebras $ C_0(K) $ for the class $ \KIII $. 
This case seems to be worse (from the point of view of the behavior of $\s$-compact subsets), but the situation is better than in cases $ \KII $ and $ \KIV $: we can find an example of an algebra for which there exists a standard frame (example \ref{iii_good}), and hence the $ \A $-compactness criterion is satisfied (one can say that in this context algebraic properties ``are not monotonic with respect to topological properties''). On the other hand, there is an example for which the criterion does not hold (and so there is no standard frame), and there is also no non-standard frame (example \ref{iii_bad}, theorem \ref{iii_bad_th}). However, there is also an intermediate example, for which there is no standard frames, but a non-standard one exists (example \ref{iii_bg}).

In \S4 it is established that for the class $ \KIV $ there are no non-standard frames in the algebra $ C_0(K) $. Earlier (in \cite{Fuf2021faa}) it was proved that for this class the Troitsky's criterion never holds, which also means that there are no standard frames.

Author is grateful to E.V.Troitsky, V.M. Manuilov, K.L. Kozlov, A.Ya. Helemskii and A.I.Korchagin for helpful discussions.

\section{Preliminaries and properties of $\KI$}

Let us recall basic notions and facts about
Hilbert $C^*$-modules and operators in them, which one can find in 
\cite{Lance},\cite{MTBook},\cite{ManuilovTroit2000JMS}.

\begin{dfn}
\rm
A (right) pre-Hilbert $C^*$-module over a $C^*$-algebra $\A$
is an $\A$-module equipped with an $\A$-\emph{inner product}
$\<.,.\>:\M\times\M\to \A$ being a sesquilinear form on the
underlying linear space such that, for any $x,y\in\M$, $a\in\A$ :
\begin{enumerate}
\item $\<x,x\> \ge 0$;
\item $\<x,x\> = 0$ if and only if $x=0$;
\item $\<y,x\>=\<x,y\>^*$;
\item $\<x,y\cdot a\>=\<x,y\>a$.
\end{enumerate}

A \emph{Hilbert $C^*$-module} is a pre-Hilbert $C^*$-module over $\A$ , which 
is complete w.r.t. its norm $\|x\|=\|\<x,x\>\|^{1/2}$.

A pre-Hilbert $C^*$-module $\M$ is called \emph{countably generated}
if there is a countable collection of its elements such that 
 their $\A$-linear combinations are dense in $\M$.

The \emph{Hilbert sum} of Hilbert
$C^*$-modules in the evident sense will be denoted by $\oplus$ .
\end{dfn}

\begin{dfn}\label{dfn:operators}
\rm
An \emph{operator} is a bounded $\A$-homomorphism.
An operator having an adjoint (in an evident sense) is 
\emph{adjointable} (see \cite[Section 2.1]{MTBook}).
We will denote the Banach space of all operators
$F: \M\to \cN$ by  ${\mathbf{L}}(\M,\cN)$
and the Banach space of adjointable operators
by ${\mathbf{L}}^*(\M,\cN)$.  
\end{dfn}

\begin{dfn}\label{dfn:Acompact}
\rm
Denote by $\theta_{x,y}:\M\to\cN$, where $x\in\cN$ 
and $y\in\M$, an \emph{elementary} $\A$-\emph{compact} operator,
which is defined by formula $\theta_{x,y}(z):=x\<y,z\>$.
Then the Banach space ${\mathbf{K}}(\M,\cN)$
of $\A$-\emph{compact} operators
is the closure of the subspace generated by all 
elementary $\A$-compact operators in ${\mathbf{L}}(\M,\cN)$.
\end{dfn}

To define a uniform structure, that is, the system of pseudometric, we need to remind the notions of the multiplier algebra and multiplier module (see \cite{BakGul2002}, \cite{BakGul2004} for more details).

Recall that $M(\A)$ is a $C^*$-algebra of multipliers of $C^*$-algebra $\A$ (see \cite[Chapter 2]{Lance}, for example), $M(\A)=\A$ if $\A$ is unital. 
Also, if $K$ is locally compact Hausdorff topological space and $\A=C_0(K)$
then $M(\A)=C_b(K)$ --- the algebra of all continuous bounded functions on $K$.

For every Hilbert $\A$-module $\cN$ there exists a Hilbert $M(\A)$-module $M(\cN)$ (which is called the multiplier module of the module $\cN$) containing $\cN$ as an ideal submodule associated with $\A$, i.e. $\cN=M(\cN)\A$. Moreover, $\<x,y\>\in\A$ holds for any $x\in\cN$, $y\in M(\cN)$. $M(\cN)=\cN$ if the algebra $\A$ is unital.
Also, since each element of $x\in\cN$ can be represented as $y\cdot a$ for some $y\in\cN$, $a\in\A$ (see \cite[ 1.3.10]{MTBook}), the module $\cN$ and any its submodule can be considered as $M(\A)$-modules.

If we consider $\cN=C_0(K)$ as a module over itself, then, as in the case of the multiplier algebra, $M(C_0(K))=C_b(K)$.

The uniform structures on submodules of $\cN$ are defined as follows (see \cite{Fuf2021faa} and \cite{Troitsky2020JMAA} for details).

\begin{dfn}\label{dfn:admissyst}
\rm
Consider a Hilbert $C^*$-module $\cN$ over $\A$. A countable system $X=\{x_{i}\}$ of elements of the multiplier module $M(\cN)$ is called (outer) \emph{admissible} for a (possibly not closed) submodule $\cN^0\subseteq \cN$
(or outer $\cN^0$-\emph{admissible}), if
\begin{enumerate}
\item[1)] 
 for each $x\in\cN^0$
the series $\sum_i \<x,x_i\>\<x_i,x\>$ is convergent in $\A$;
\item[2)]
its sum is bounded by $\<x,x\>$, that is, $\sum_i \<x,x_i\>\<x_i,x\> \le \<x,x\>$; 
\item[3)]$\|x_i\|\le 1$ for any $i$.
\end{enumerate}

\end{dfn}

\begin{dfn}\rm
Denote by $\F$ a countable collection $\{\f_1,\f_2,\dots\}$
of states on $\A$ (i.e. positive linear functionals of norm 1). For each pair $(X,\F)$ 
with an outer $\cN^0$-admissible $X$,
consider a non-negative function defined by
the equality
$$
\nu_{X,\F}(x)^2:=\sup_k 
\sum_{i=k}^\infty |\f_k\left(\<x,x_i\>\right)|^2,\quad x\in \cN^0. 
$$
It can be checked that this is a seminorm on the module $\cN^0$.
Denote the system of all these functions by 
$\mathbb{OSN}(\cN,\cN^0)$. Also we will write $(X,\F)\in 
\mathbb{OA}(\cN,\cN^0)$
for pairs with outer admissible $X$. 
\end{dfn}

\begin{dfn}\label{dfn:pseme}
Consider for $(X,\F)\in \mathbb{OA}(\cN,\cN^0)$
the following function
$d_{X,\F}:\cN^0\times \cN^0\to [0,+\infty)$
$$
d_{X,\F}(x,y)^2:=\nu_{X,\F}(x-y)^2=
\sup_k 
\sum_{i=k}^\infty |\f_k\left(\<x-y,x_i\>\right)|^2,\quad x,y\in \cN^0. 
$$
We will write $d_{X,\F}\in \mathbb{OPM}(\cN,\cN^0)$. 
\end{dfn}

Evidently,
$d_{X,F}$ are \emph{pseudometrics} in sense of  \cite[Definition 2.10]{Troitsky2020JMAA} (and \cite[Chapter IX, \S 1]{BourbakiTop2}), so they form a uniform structure.

If $X$ contains only elements of the module $\cN$, the word ``outer'' is not used, and in this case one may write $\mathbb{SN}$, $\mathbb{A}$ and $\mathbb{PM}$ instead of $\mathbb{OSN}$, $\mathbb{OA}$ and $\mathbb{OPM}$.

The definition of \emph{totally bounded} sets 
for the uniform structure under consideration
(or for the system $\mathbb{OPM}(\cN,\cN^0)$) takes the following form.

\begin{dfn}\label{dfn:totbaundset}
\rm
A set $Y\subseteq \cN^0 \subseteq  \cN(\subseteq  M(\cN))$ is \emph{totally bounded}
with respect to this
uniform structure, if for any $(X,\F)$, 
where $X \subseteq  M(\cN)$ is outer $\cN^0$-admissible,
and any 
$\e>0$ there exists a finite collection $y_1,\dots,y_n$
of elements of $Y$ such that the sets
$$
\left\{ y\in Y\,|\, d_{X,\F}(y_i,y)<\e\right\}
$$  
form a cover of $Y$. This finite collection is called an 
$\e$\emph{-net in $Y$ for} $d_{X,\F}$.

If so, we will say that $Y$ is \emph{externally} $(\cN,\cN^0)$-\emph{totally bounded} (or $(M(\cN),\cN^0)$-totally bounded).
\end{dfn}

In these terms the $\A$-compactness of operators for some class of modules can be describe by following:

\begin{teo}(\cite[Theorem 3.5]{TroitFuf2020})
Suppose, $\M$, $\cN$ and ${\mathcal K}$ are Hilbert $\A$-modules, $\cN\oplus{\mathcal K}\cong\bigoplus\limits_{\l\in\Lambda}\dot{\A}$ for some $\L$,
$F:\M\to\cN$ is an adjointable operator
and $F(B)$ is $(\cN,\cN)$-totally bounded, where $B$ is the unit ball of $\M$.
Then $F$
is $\A$-compact as an operator from $\M$ to $\cN$.
\end{teo}

The inverse statement holds for arbitrary modules (\cite[Theorem 2.4]{TroitFuf2020}). For the case when $\cN$ is a countably generated module a similar result was stated and proved as criterion of $\A$-compactness by E.V. Troitsky (\cite[Theorem 2.13]{Troitsky2020JMAA}).
%So the problem here is to find conditions when this  

Evidently $(M(\cN),\cN^0)$-totally bounded set is $(\cN,\cN^0)$-totally bounded, so all results which state that $(\cN,\cN)$-totally boundedness of the image of the unit ball implies $\A$-compactness of corresponding adjointable operator are still valid if we consider $(M(\cN),\cN)$-totally boundedness.

Let us recall a notion of a frame in Hilbert $C^*$-module (see, e.g., \cite{FrankLarson1999}, \cite{FrankLarson2002}).
Among all frames, standard ones are also considered.

\begin{dfn}\label{dfn:fr}
Let $\cN$ be a Hilbert $C^*$-module over an unital $C^*$-algebra $\A$ and $J$ be some set. A family $\{x_j\}_{j\in J}$ of elements of $\cN$ is said to be a standard frame in $\cN$ if there exist positive constants $c_1, c_2$ such that
for any $x\in\cN$
the series 
$\sum\limits_j \<x,x_j\>\<x_j,x\>$
converges in norm in $\A$
and the following inequalities hold:

 $$c_1\<x,x\>\le \sum\limits_j \<x,x_j\>\<x_j,x\>\le c_2\<x,x\>.$$
A frame is called tight if $c_1=c_2$ and normalized if $c_1=c_2=1$.
If the series converges only in the ultraweak topology (also known as $\s$-weak) to some element of the universal enveloping von Neumann algebra $\A''$, then the frame is said to be non-standard. 
Unlike the case of a standard frame, in this case
the number of nonzero elements of the series can be uncountable, the convergence in this case is considered as the convergence of a net consisting of all finite partial sums (see remarks before \cite[1.2.19]{KadRin1} and \cite[5.1.5]{KadRin1}).
We will write just ``frame'' if it is at least non-standard.
\end{dfn}

If the algebra $\A$ is not unital, then $\cN$ can be considered as a module over its unitalization $\dot{\A}$ and frame can be defined in $\cN$ as in an $\dot{\A}$-module.
Further we will assume that frames are defined in this way.

\begin{rk}
For a system $\{x_j\}_{j\in J}$ there is a connection between the so-called reconstruction formula and the property of being a frame: if for any $x\in\cN$ it holds that
$x=\sum\limits_{j\in J}x_j\<x_j,x\>$,
then $\{x_j\}_{j\in J}$ is a normalized frame, and the convergence takes place in the same sense (see \cite[Example 3.1]{FrankLarson2002})
\end{rk}

\begin{rk}\label{rk_bes}
If there exists a positive constant $c$ such that for any $x\in\cN$ and
for any partial sum of the considered series the inequality
$\sum\limits_j \<x,x_j\>\<x_j,x\>\le c\<x,x\>$ holds (that is, the right side of the inequality for the frame), then the system $\{x_j\} _{j\in J}$ is called the Bessel system. Due to \cite[2.4.19]{BratRob} the series with respect to the Bessel system always converges in the ultrastrong topology and, as a consequence, also in the ultraweak one.
\end{rk}

Recall the following characterization of non-standard frames:

\begin{lem}\label{fr_cr}
(\cite[Proposition 3.1]{HLi2010}) 
A system $\{x_j\}_J$ is a frame in $\cN$ if and only if there exist positive constants $c_1, c_2$ such that for any $x\in\cN$ and any state $\f$ on $\A$
% (i.e. a positive linear functional of norm 1) 
the following inequalities hold:

$$
%\begin{equation}%\label{framein}
c_1\f(\<x,x\>)\le
\sum\limits_j \f(\<x,x_j\>\<x_j,x\>)\le
c_2\f(\<x,x\>)
%\end{equation}
$$

\end{lem}

From this property and the definition of a standard frame it is obvious that any frame in the algebra $C_0(K)$ as a module over itself must separate points of the space $K$.

In our context the main structural result of the frame theory is the following:
from the results of Frank and Larson (\cite[3.5, 4.1 and 5.3]{FrankLarson2002}, see also \cite[Theorem 1.1]{HLi2010}) it follows that the Hilbert $C^*$-module
$\cN$ over the $C^*$-algebra $\A$ can be represented as an orthogonal direct summand in the standard module of some cardinality over the unitalization algebra $\bigoplus\limits_{\l\in\Lambda}\dot{\A}$ if and only if there exists a standard frame in $\cN$.

Kasparov's stabilization theorem (\cite{Kasp}, or \cite[Theorem 1.4.2]{MTBook}) implies that every countably generated module has a standard frame. The construction of a module that does not have a standard frame will show that the stabilization property of the Kasparov type is not satisfied for this module.

Note that frames are well-defined even for the case when the algebra is not unital,% (more precisely, there is no difference to co), 
but to use this stabilization property we need to take its unitalization.

\begin{rk}
If an algebra $\A$ has a standard frame as a module over itself, then any $\A$-module $\cN$ which can be represented as an orthogonal direct summand in the standard module $\bigoplus\limits_{\l\in\Lambda}{\A}$ also has a standard frame (and hence can be represented as an orthogonal direct summand in the standard module $\bigoplus\limits_{\l\in\Lambda'}\dot{\A}$). For some special case it was noted in \cite[remark 3.6]{TroitFuf2020}.
\end{rk}

Let us recall now some preliminaries from topology.

For a locally compact Hausdorff topological space $K$ we denote by $\a{K}=K\cup\{t_\infty\}$ its one-point compactification (see \cite[29.1]{Munkr}). We will also often use the fact that the $C^*$-algebra $C_0(K)$ of continuous functions vanishing at infinity (see \cite[436I]{Frem4}) is isomorphic to an ideal in $C^*$ -algebra $C(\a{K})$ consisting of functions vanishing at the point $t_\infty$ (\cite[Lemma 3.44]{Weaver}).
We denote by $\be K$ the Stone-\v{C}ech compactification of $K$; it is known that $C_b(K)\cong C(\be K)\cong M(C_0(K))$ (see \cite[Chapter 1]{Walker}).

\begin{lem}\label{extrest} (\cite[Corollary 1.3]{Fuf2021faa})
Let $K$ be a locally compact Hausdorff space, $A$ be a closed subset of $K$ and
$f\in C_0(A)$,
then $f$ can be extended to a function from $C_0(K)$. Moreover, $f$ is bounded and the extended function can be chosen to be bounded by the same constant.
In particular, for every compact set $K'\subset K$ there is a function $g\in C_0(K)$ such that $g=1$ on $K'$ and $|g|\le1$ on $K$ .
\end{lem}

We need the following useful examples of topological spaces:
$[0,\om_1]$ --- the space of all ordinals $\a$ such that $\a\le\om_1$ with order topology, where $\om_1$ is the first uncountable ordial; this space is uncountable. Also, $[0,\om_1)\in\KIV$.
$[0,\om_0]$ is the space of all ordinals $\a$ such that $\a\le\om_0$ with order topology, where $\om_0$ is the first infinite ordial (this space is homeomorphic to $\{\frac{1}{n}\}_{n\in\N}\cup \{0\}$, or $\N\cup\{\infty\}$). See \cite{Fuf2021faa} or \cite[chapters VI-VII]{KurMos} and \cite[3.1.27]{Engel} for details.

The following well-known for specialists statement is useful for our goal.

\begin{lem}
The $C^*$-algebra $C_0(K)$ is $\s$-unital if and only if $K$ is $\s$-compact.
\end{lem}

\begin{proof}
Recall that $\s$-unitality is equivalent to the existence of a strictly positive element, in our case --- an everywhere positive function.

If the algebra $C_0(K)$ is $\s$-unital, then it contains an everywhere positive function $f$. The sets $\{t\in K:|f(t)|\ge\frac{1}{n}\}$ are compact (as closed subsets of such corresponding compact sets outside of which $|f(t)|<\frac {1}{n}$), and therefore form a countable cover of $K$ by compacts.

Conversely, if there are compact sets $K_n$ such that $K=\bigcup\limits_{n\in\N}K_n$, then there exist functions $f_n\in C_0(K)$ such that $f_n(t)=1 $ on $K_n$ and $|f_n(t)|\le1$ everywhere on $K$. The function $\sum\limits_{n\in\N}\frac{1}{2^n}|f_n(t)|$ is everywhere positive on $K$, since for any point there is a compact set $K_n$ in which it is contained.
\end{proof}

\begin{lem}\label{frp1}
Let $K$ be a locally compact, not $\s$-compact space. Then for any countable family of functions $\{f_n\}\subset C_0(K)$ the set $F=\{t\in K: f_n(t)=0\,\, \forall n\in\N\}$ is non-empty, and even uncountable.

In particular, a countable family of functions cannot separate points of the space $K$.
\end{lem}

\begin{proof}
Indeed, if the set $F$ is countable and equals to $\{z_n\}_{n\in \N}$, then for every $n\in\N$ there exists a function $g_n\in C_0(K)$ such that $ g_n(z_n)=1$, $|g_n(t)|\le1$, %(see \cite[1.3]{Fuf2021faa}), 
and then the function $g(t)=\sum\limits_{n\in\N}\frac{1}{2^n}\frac{|f_n(t)|}{1+|f_n(t)|}+\sum\limits_{n\in\N}\frac{1}{ 2^n}|g_n(t)|$ is everywhere positive on $K$, which contradicts the fact that $K$ is not $\s$-compact.%lack of sigma-compactness.

The proof is also valid for the cases when $F$ is finite or empty.
\end{proof}

\begin{lem}\label{frp}
Let $K$ be a locally compact, not $\s$-compact space, $\{K_n\}_{n\in\N}$ be a countable collection of compact sets from $K$. Then for any countable family of functions $\{f_n\}\subset C_0(K)$ the set $F=\{t\in K\setminus \bigcup\limits_{n\in\N}K_n: f_n(t)=0 \,\, \forall n\in\N\}$ is non-empty, and even uncountable.

\end{lem}

\begin{proof}
Similarly to the previous lemma, consider the function $g(t)=\sum\limits_{n\in\N}\frac{1}{2^n}\frac{|f_n(t)|}{1+|f_n(t) |}+\sum\limits_{n\in\N}\frac{1}{2^n}|g_n(t)|+\sum\limits_{n\in\N}\frac{1}{2^ n}h_n(t)$, where $h_n\in C_0(K)$ is a non-negative function which equals to 1 on $K_n$ and bounded by 1.
\end{proof}

It is known (see \cite[Example 3.5]{FrankLarson2002}) that every $\s$-unital algebra as a module over itself has a normalized standard frame.
Let us construct an example of a non-$\s$-unital algebra that has a frame with the same properties.

Recall that $\A=c_0-\sum\limits_{\l\in\L}\A_\l$ is a $c_0$-direct sum of the algebras $\A_\l$ (see \cite[\S 1.4]{Averson} or \cite[\S 1]{FrM2020}), the elements $x=(x_\l)_{\l\in\L}$ of this sum are enumerated by indices from $\L$ sets such that $x_\l\in\A_\l$, 
there are at most countable set $\{x_{\l_m}\}_{m\in\N}$ of non-zero elements and $\lim\limits_{m\to\infty}||x_{\l_m}||_{\A_{\l_m}}
=0$.
The norm in this algebra is given by the formula $||x||_\A=\sup\limits_{\l\in\L}||x_\l||_{\A_\l}$.
In fact, this algebra is obtained by completing an algebra whose elements are non-zero only for a finite number of indices.

\begin{teo}\label{iii_stfr}
Let $\A=c_0-\sum\limits_{\l\in\L}\A_\l$,
where each $\A_\l$ has a normalized standard frame for which the reconstruction formula is valid (for example, due to \cite[Proposition 2.3]{Asadi2016}, when each $\A_\l$ is $\s$-unital).
Then $\A$ as a $\dot{\A}$-module has a normalized standard frame.
\end{teo}

\begin{proof}
Denote by $\{x_j\}_{j\in J_\l}$ the frame in $\A_\l$.
For each $\l\in\L$, the elements of the frame can be considered as elements of the entire algebra $\A$ by extending them by zero
outside the corresponding index. Let us show that $\{x_j\}_{j\in \bigcup\limits_{\l\in\L}J_\l}$ is a frame in $\A$.

Let $x\in\A$. Fix arbitrary $\e>0$. There exists an element $x^m\in\A$ which is non-zero only in a finite number of indices $\l_1,\dots,\l_m$ such that $||x-x^m||<\e$. For each $\l_k$, $k=1,\dots,m$, there exists a finite set $\{x_j\}_{j\in J_{\l_k}'}$ of elements of the frame such that
$$||x^m_{\l_k}-\sum\limits_{j\in J_{\l_k}'}x_j\<x_j,x^m_{\l_k}\>||<\e,$$ hence
$$||x^m-\sum\limits_{j\in \bigcup\limits_{k=1}^mJ_{\l_k}'}x_j\<x_j,x^m\>||<\e$$ and
$$||x-\sum\limits_{j\in \bigcup\limits_{k=1}^mJ_{\l_k}'}x_j\<x_j,x^m\>||<2\e.$$ 
So we have
$x=\sum\limits_{j\in J}x_j\<x_j,x\>$, 
and the series converges in norm in $\A$. Hence 
$\<x,x\>=\sum\limits_{j\in J}\<x,x_j\>\<x_j,x\>$.

\end{proof}

\begin{cor}
Since $\A$ has a frame, it can be represented as an orthogonal direct summmand in the standard module
$\bigoplus\limits_{\l\in\Lambda}\dot{\A}$
of some cardinality, and so, by \cite[Theorem 3.5]{TroitFuf2020}, it satisfies the $\A$-compactness criterion.

Moreover, every $\A$-module which can be represented as an orthogonal direct summmand in the standard module $\bigoplus\limits_{\l\in\Lambda'}{\A}$ also can be represented
in the standard module
$\bigoplus\limits_{\l\in\Lambda}\dot{\A}$
of some cardinality, and hence it satisfies the $\A$-compactness criterion too.
\end{cor}

\begin{rk}\label{iii_dop}
If $\A$ also is a commutative algebra, then $\A\cong c_0-\sum\limits_{\l\in\L}C_0(K_\l)\cong C_0(\bigsqcup\limits_{\l\in \L}K_\l)$, and $C_0(K_\l)$ is $\s$-unital if and only if $K_\l$ is $\s$-compact; if $\L$ is uncountable, then $K=\bigsqcup\limits_{\l\in\L}K_\l$ is not $\s$-compact, because a compact set in $K$ intersects only a finite number of $K_\l$. If $\A_\l$ is a unital algebra, then as a frame in $\A_\l$ we can take just one element, the identity of the algebra (it corresponds to a function which identically equals to one in the commutative case).
\end{rk}

 \section{The properties of $\KII$}

Let us introduce several examples of spaces from the class $\KII$.

\begin{ex}\label{rat_seq}
$K$ --- the set of real numbers with rational sequence topology (\cite[\S 65]{Exampl}). Moreover, it is separable.
\end{ex}

\begin{ex}\label{beN}
$K=\be\N\setminus\{t'\}$, where $t'\in\be\N\setminus\N$ --- the Stone-\v{C}ech compactification of natural numbers without an arbitrary
point from the 
growth.
It is not $\s$-compact since \cite[9.6]{GilJer}, and obviously it is separable.
More generally, we can take instead of $\N$ any separable non-compact space (or just $\s$-compact, but we can lost separability).
\end{ex}

\begin{ex} 
$K=\a{P}\times [0,\om_0]\setminus\{(p_\infty,\om_0)\}$, where $P$ is a locally compact, non-$\s$-compact Hausdorff space, $\a{P}=P\cup\{p_\infty\}$ --- its one-point compactification.
\end{ex}

\begin{teo}\label{ii_nostfr}
Let $K\in\KII$. Then there is no standard frame in the $\dot{C_0}(K)$-module $C_0(K)$.
\end{teo}

\begin{proof}
Assume that there exists a frame $\{x_j\}_{j\in J}$ of elements from $C_0(K)$. It contains an uncountable number of nonzero elements, since countable cannot separate the points of $K$.

Let $\{K_n\}_{n\in\N}$ be a sequence of compact sets in $K$ such that $\overline{\bigcup\limits_{n\in\N}K_n}=K$.

For every $n\in\N$ there exists a function $g_n\in C_0(K)$ such that $g_n(t)=1$ on $K_n$, $|g_n(t)|\le1$ on $K$.
There is a non-empty at most countable set $\{x_j\}_{j\in J_n}$ of elements of the frame such that $x_j(t)\ne0$ identically on $K_n$ because the series $\sum\limits_j \<g_n,x_j\>\<x_j,g_n\>(t)=\sum\limits_j |x_j(t)|^2$ converges uniformly on $K_n$. That is, if $j\in J\setminus J_n$, then $x_j(t)=0$ on $K_n$.

Hence, there is at most countable set $\{x_j\}_{j\in \bigcup\limits_{n\in\N}J_n}$ such that $x_j(t)\ne0$ identically on $\bigcup\limits_{n\in\N}K_n$. Hence, for every $j\in J\setminus \bigcup\limits_{n\in\N}J_n$ we have $x_j(t)=0$ for $t\in\bigcup\limits_{n\in\N} K_n$, but due to the fact that $\overline{\bigcup\limits_{n\in\N}K_n}=K$ it also holds for $t\in K$. Hence, only a countable set of frame elements is not identically zero.
A contradiction.
\end{proof}

\begin{cor}
From this result it follows that $C_0(K)$ cannot be represented as an orthogonal direct summand of a standard module $\bigoplus\limits_{\l\in\Lambda}\dot{C_0}(K)$.
\end{cor}

Let now $K$ be 
moreover separable, that is, finite sets can be taken as compacts in the definition of $\KII$.
Spaces from examples \ref{rat_seq} or \ref{beN} can be taken as such spaces.
Let us show that in this case 
non-standard frames also don't exist.

\begin{teo}\label{ii_nofr}
Let $K\in\KII$ and $K$ is separable.
Then there is no frames in the $\dot{C_0}(K)$-module $C_0(K)$.
\end{teo}

\begin{proof}

Assume that there exists a frame $\{x_j\}_{j\in J}$ of elements from $C_0(K)$. It must contain an uncountable number of nonzero elements, since countable cannot separate the points of $K$.

Let $\{t_n\}_{n\in\N}$ be a countable dense subset of $K$.

For each $n\in\N$ there exists a function $g_n\in C_0(K)$ such that $g_n(t_n)=1$, $|g_n(t)|\le1$ on $K$.
There is a non-empty at most countable set $\{x_j\}_{j\in J_n}$ of elements of the frame such that $x_j(t_n)\ne0$,
because
taking $x=g_n$ and $\f$ --- evaluation at the point $t_n$, due to lemma \ref{fr_cr} we have that
the series $\sum\limits_j \<g_n,x_j\>\<x_j,g_n\>(t)=\sum\limits_j |x_j(t)|^2$ converges at the point $t_n$. That is, if $j\in J\setminus J_n$, then $x_j(t_n)=0$.

Hence, there is at most countable set $\{x_j\}_{j\in \bigcup\limits_{n\in\N}J_n}$ such that $x_j(t)\ne0$ identically on $\bigcup\limits_{n\in\N}\{t_n\}$. Hence, for every $j\in J\setminus \bigcup\limits_{n\in\N}J_n$ we have that $x_j(t)=0$ for $t\in\bigcup\limits_{n\in\N} \{t_n\}$, but due to the fact that $\overline{\bigcup\limits_{n\in\N}\{t_n\}}=K$ it also holds for $t\in K$. Hence, only a countable set of frame elements is identically zero.
A contradiction.
\end{proof}

Despite the previous two theorems, it is possible to construct an example of a space $K\in\KII$ such that a non-standard frame exists in $C_0(K)$.

\begin{ex}\label{ii_exfr}
Let $P$ be a non-$\s$-compact space such that $C_0(P)$ has a normalized frame $\{u_\be\}_{\be\in B}$. %, whether standard or not.
We know that such a space exists (see remark \ref{iii_dop}).

Take $K=\a{P}\times [0,\om_0]\setminus\{(p_\infty,\om_0)\}$ and define $y_\be\in C_0(K)$, $\be\in B$, by the formula $y_\be(p,n)=u_\be(p)$, where $(p,n)=t\in K$, i.e. $\{y_\be\}_{\be\in B}$ is a ``copying'' of functions $\{u_\be\}_{\be\in B}$ on each ``row'' ${ P}\times\{n\}$, $n\in[0,\om_0]$. Also consider $\{w_n\}_{n\in[0,\om_0)}$ such that $w_n=1$ on $\a{P}\times\{n\}$ and $w_n$ vanishes outside $\a{P}\times\{n\}$
(obviously, $w_n\in C_0(K)$ for any $n\in[0,\om_0)$, since every $\a{P}\times\{n\}$ is a clopen set).
Define $\{x_j\}_{j\in J}=\{y_\be\}_{\be\in B}\cup\{w_n\}_{n\in[0,\om_0)}$ .
\end{ex}

\begin{teo}\label{ii_fr}
In example \ref{ii_exfr} the system $\{x_j\}_{j\in J}$ is a (non-standard) frame in $C_0(K)$.
\end{teo}

\begin{proof}
Take an arbitrary $x\in C_0(K)$. For any partial sum of the series $\sum\limits_j \<x,x_j\>\<x_j,x\>(t)=\sum\limits_j|x(t)|^2|x_j(t)|^2$ we have $\sum\limits_j|x(t)|^2|x_j(t)|^2\le2|x(t)|^2=2\<x,x\>(t)$ (actually, if $p=p_\infty$ or $n=\om_0$, where $t=(p,n)$, then two can be replaced by one), which means that the series converges in the ultrastrong topology.

Hence we obtain the right side of the inequality
% (\ref{framein}) 
of lemma \ref{fr_cr} with $c_2=2$ for any state $\f$ (in particular, the corresponding series converges).
Indeed, if for some state $\f$ we have $\sum\limits_j \f(\<x,x_j\>\<x_j,x\>)>
2\f(\<x,x\>)$,
then there exists some partial sum of the series for which it also holds that
$\sum\limits_{j\in J'} \f(\<x,x_j\>\<x_j,x\>)>
2\f(\<x,x\>)$. And then we have $\f(2\<x,x\>-\sum\limits_{j\in J'} \<x,x_j\>\<x_j,x\>)<0$, which contradicts the inequality
$2\<x,x\>\ge\sum\limits_{j\in J'} \<x,x_j\>\<x_j,x\>$ and positivity of $\f$.

Let us show that the left side holds too.

Let $\f$ be a state on $C_0(K)$, that is, there is a Radon measure $\mu$ on $K$ such that $\f(x)=\int\limits_{K}x( t)d\mu(t)$ (see \cite[436K]{Frem4}).
Then $\mu$ can be represented as a sum of the Radon measures $\mu_1$ which support is $P\times \{\om_0\}$ and $\mu_2$ which support is $\a{P}\times [0,\om_0) $.
Then for any $x\in C_0(K)$ we have
$\f(\<x,x\>)=\int\limits_{P\times \{\om_0\}}|x(t)|^2d\mu_1(t)+\int\limits_{\a{ P}\times [0,\om_0)}|x(t)|^2d\mu_2(t)$.
The representation of a measure as a sum corresponds to the representation of $\f$ as a sum of states $\f_1$
and $\f_2$.

Identify the restrictions of $y_\be(p,n)$ on $P\times \{\om_0\}$ with $u_\be(p)$, then $\{u_\be\}_{\be\in B }$ is a normalized frame in $C_0(P\times \{\om_0\})$.
%We have that f
For
the restriction of $x$ to $P\times \{\om_0\}$ (and hence for $x$ itself, since $\mu_1$ vanishes outside this subset) we have
$\f_1(\<x,x\>)\le
\sum\limits_{\be\in B}\f_1(\<x,u_\be\>\<u_\be,x\>)
=\sum\limits_{j}\f_1(\<x,x_j\>\<x_j,x\>)
$
($\f_1$ can be considered as a state on $C_0(P\times \{\om_0\})$).

$P\times \{\om_0\}$ is a closed set in $K$, so $\a{P}\times [0,\om_0)$ is open and hence it is Borel set. Therefore $\mu_2$ is a Radon measure on $\a{P}\times [0,\om_0)$ (\cite[416R(b)]{Frem4}).
By the monotone convergence theorem (see, for example, \cite[Theorem 2.25]{Weaver} or \cite[Proposition 8.7(b)]{Yeh})
we have
$$
\int\limits_{\a{P}\times [0,\om_0)}|x(t)|^2d\mu_2(t)=
\int\limits_{\a{P}\times [0,\om_0)}\sum\limits_{n=1}^\infty|x(t)|^2|w_n(t)|^2d\mu_2(t)=
\sum\limits_{n=1}^\infty\int\limits_{\a{P}\times [0,\om_0)}|x(t)|^2|w_n(t)|^2d\mu_2(t),%=
$$
hence

\begin{multline*}
\f_2(\<x,x\>)=\int\limits_{\a{P}\times [0,\om_0)}|x(t)|^2d\mu_2(t)=
\sum\limits_{n=1}^\infty\int\limits_{\a{P}\times [0,\om_0)}|x(t)|^2|w_n(t)|^2d\mu_2(t)=\\
=\sum\limits_{n=1}^\infty\f_2(\<x,w_n\>\<w_n,x\>)\le
\sum\limits_{j}\f_2(\<x,x_j\>\<x_j,x\>).
\end{multline*}

By summing up the obtained inequalities, we have
$
\f(\<x,x\>)\le
\sum\limits_{j}\f(\<x,x_j\>\<x_j,x\>)
$. 

Thus, for every state $\f$ on $C_0(K)$ and every $x\in C_0(K)$
the inequalities $\f(\<x,x\>)\le\sum\limits_j \f(\<x,x_j\>\<x_j,x\>)\le2\f(\<x,x\> )$ hold, so $\{x_j\}_{j\in J}$ is a frame in $C_0(K)$ with constants 1 and 2.
\end{proof}

\begin{rk}
It is clear that the proof of the previous theorem is still valid if instead of normalized frame in $C_0(P)$ we take a frame with arbitrary frame constants.
\end{rk}

The existence of a standard frame is a sufficient but not necessary condition for the $\A$-compactness criterion to be satisfied.
Therefore we cannot assert that for any topological space from the
class $\KII$
the criterion is not valid, 
but we can construct an example of a space (more precisely, some subclass of spaces) for which the criterion actually fails.

First, let us consider several properties of topological spaces (it is assumed everywhere that $K$ is a locally compact Hausdorff space).

1) $\be K=\a K$ (in other words, every continuous bounded function has a limit at infinity).

2) Any $\s$-compact subset of $K$ is precompact in $K$ (that is, $K\in\KIV$).

2.1) Any continuous function which tends to zero at infinity is constant outside some compact $K'$.

2.2) Any continuous function that has a limit at infinity is constant outside some compact $K'$.

3) Any continuous function on $K$ is constant outside some compact $K'$.

Let's observe the relationships between these properties.

\begin{lem}
The properties 2), 2.1), 2.2) are equivalent.
\end{lem}

\begin{proof}
Indeed, 2.1) obviously follows from 2.2).

Let's prove 2.2) if 2.1) is true.
Let $f$ be a continuous function which has a limit at infinity equals to $f_\infty$. Then the continuous function $f-f_\infty$ tends to zero at infinity, so it vanishes outside some compact set, so the function $f$ is constant and equals to $f_\infty$ outside the same compact set.

Let's prove 2) if 2.1) is true.
Let $\{K_n\}_{n\in\N}$ be a sequence of compact sets in $K$. For every $n\in\N$ there exists a function $g_n\in C_0(K)$ such
that $g_n(t)=1$ on $K_n$, $|g_n(t)|\le1$ on $K$. Define the function $g\in C_0(K)$ by formula $g(t)=\sum\limits_{n\in\N}\frac{1}{2^n}g_n(t)$. It tends to zero at infinity, so it vanishes outside some compact $K'$, and moreover, it is nonzero at $\bigcup\limits_{n\in\N}K_n$, so $\bigcup\limits_ {n\in\N}K_n\subset K'$. Q.E.D.

Implication 2) $\Rightarrow$ 2.1) was proved in \cite[Lemma 1.5]{Fuf2021faa}.

\end{proof}

Obviously, 3) implies 2). The converse, in general, is false; it suffices to consider a disjoint union of sets with property 2), for example, $[0,\om_1)$. The same example shows that 2) does not imply 1).

From 3) it also obviously follows that 1) holds. The converse is not true, as will follow from the example we will construct. Also, this example will not satisfy 2). Also, this example will represent a class of spaces for which the criterion of $\A$-compactness fails.

\begin{ex}\label{ii_exnocr}
Take $K=\a{P}\times [0,\om_0]\setminus\{(p_\infty,\om_0)\}$, where $P\in\KIV$ (that is, $P$ satisfies the property 2)),
$\a{P}=P\cup\{p_\infty\}$ --- its one-point compactification.
\end{ex}

A special case of this construction is the deleted Tychonoff plank (\cite[\S 87]{Exampl}) if $P=[0,\om_1)$, $\a P=[0,\om_1]$.

This space does not satisfy 2) (and, as a consequence, does not satisfy 3)), since it contains a countable dense family of compact sets $\{\a P\times \{n\}\}_{n\in\N }$. Let us show that it satisfies 1)
(this generalizes the properties of the deleted Tychonoff plank, see \cite[8.20]{GilJer}).

\begin{teo}
Let $K$ be the space from the example \ref{ii_exnocr}. Then $\a K=\be K$, that is, every continuous bounded function has a limit at infinity. Moreover, $K$ is pseudo-compact, that is, every continuous function on $K$ is bounded.
\end{teo}

\begin{proof}
Let $f\in C(K)$. Then for each $n\in\N$ the restriction of $f$ to $\a P\times \{n\}$ is continuous, and the restriction of $f$ to $P\times \{n\}$ is a continuous function that has a limit at infinity. Hence, by property 2), for every $n\in\N$ there is a compact set $P_n\subset P$ such that $f$ is constant (and equals to some $p_{n,\infty}$) outside the compact $ P_n \times \{n\}\subset P\times \{n\}$. $\{P_n\}_{n\in\N}$ is a countable family of compact sets in $P$, so there exists a compact set $P'\subset P$ which contains all of them. So outside $P'\times [0,\om_0)$ the function $f$ depends only on the number $n\in\N$ and is equals to $p_{n,\infty}$ on the ``$n$th row'' $P\times\{n\}$.

Consider $(p,\om_0)\in P\times\{\om_0\}$ with $p\notin P'$. Since $f$ is continuous at $(p,\om_0)$, we have $f(p,\om_0)=\lim\limits_{n\to\infty}f(p,n)$, this limit does not depend on $p$ outside $P'$, hence the function $f(p,\om_0)$ is constant outside $P'$, so the function $f$ can be extended by continuity at the point $(p_\infty,\om_0)$ by $f(p_\infty,\om_0)=\lim\limits_{n\to\infty}f(p_\infty,n)=\lim\limits_{p\to p_\infty}\lim\limits_{n \to\infty}f(p,n)$.
Thus, we have proved the property 1).
\end{proof}

\begin{cor}
For $K$ from the example \ref{ii_exnocr} we have
$
C(K)=C_b(K)=
\dot{C_0}(K)=M(C_0(K))\cong C(\a K)
=C(\be K)
$ (except the first equality, this is also true for any space with property 1)).
\end{cor}

To show that Troitsky's theorem does not hold for $C_0(K)$ as a module over itself with such $K$, we need one more intermediate step.

\begin{teo}
A system $\{x_j\}\subset C_b(K)$ is $(C_b(K),C_0(K))$-admissible if and only if it is $(C_b(K),C_b(K)) $-admissible.
\end{teo}

\begin{proof}
An implication $\Leftarrow$ is obvious; let us prove the inverse.

Let $\{x_j\}\subset C_b(K)$ be $(C_b(K),C_0(K))$-admissible, i.e., for every $x\in C_0(K)$ we have

\begin{enumerate}
\item[1)] 
the series $\sum_i \<x,x_i\>\<x_i,x\>$ converges in norm (i.e., uniformly);
\item[2)]
its sum is bounded by $\<x,x\>$; 
\item[3)]$\|x_i\|\le 1$ for any $i$.
\end{enumerate}

Let us take an arbitrary function $x\in C_b(K)$ and show that these conditions are also satisfied for it (it suffices to show 1) and 2), obviously).

Similar to the previous proof, there exists a compact set $P'\subset P$ such that the function $x$ and all the functions $x_j$ are constant outside $P'\times[0,\om_0]$ on each ``row''. There exists $p'\in P\setminus P'$, denote $P''=P'\cup\{p'\}$.

There exists a function $g\in C_0(P)$ such that $g(p)=1$ on $P''$, $|g(p)|\le1$ on $P$.
Define
$\widetilde{x}\in C_0(K)$ by the formula $\widetilde{x}(t)=x(t)g(p)$, where $t=(p,n)\in K$. $\widetilde{x}$ satisfies conditions 1) and 2) on $K$, and hence on the set $P''\times [0,\om_0]$ on which $\widetilde{x}=x$ .
Outside $P''$ we have
$\sum\limits_i \<x,x_i\>\<x_i,x\>(p,n)=
\sum\limits_i \<x,x_i\>\<x_i,x\>(p'',n)
$. That is, outside $P''$ conditions 1) and 2) are satisfied, since they are satisfied for $p=p'$.
If the uniform convergence on each of two sets holds, then it also holds on their union; if the inequality holds on sets, then it also holds on their union. Hence, conditions 1), 2) are satisfied on the whole $K$ for any $x\in C_b(K)$. Q.E.D.
\end{proof}

\begin{cor}
A set $Y\subset C_0(K)$ is $(C_b(K),C_0(K))$-totally bounded if and only if it is $(C_b(K),C_b(K))$-totally bounded.
\end{cor}

\begin{rk}
Using \cite[4.2]{Walker} one can construct more complex examples of spaces with the described properties, by taking instead of $[0,\om_0]$ an arbitrary infinite compact set and 
choosing for $P$ instead of $\om_1$
a sufficiently large ordinal if it necessary --- to use the condition that on each ``row'' the function is eventually constant.
\end{rk}

\begin{teo}\label{ii_nocr}
The unit ball in $C_0(K)$ (and hence the image of the unit ball with respect to the identity operator $Id:C_0(K)\to C_0(K)$) is $(C_b(K),C_0(K))$- totally bounded, but the identity operator is not $C_0(K)$-compact.
\end{teo}

\begin{proof}
The unit ball in $C_0(K)$ is a subset of the unit ball in $C_b(K)$, which is $(C_b(K),C_b(K))$-totally bounded since it is the image of the unit ball with respect to the identity operator $Id:C_b(K)\to C_b(K)$, which is $C_b(K)$-compact because $C_b(K)$ is unital and it is countably generated as a module over itself. Hence, the unit ball in $C_0(K)$ is also $(C_b(K),C_b(K))$-totally bounded, and by the previous corollary it is $(C_b(K),C_0(K))$-totally bounded.

The idenitity operator is not $C_0(K)$-compact since the image of $\A$-compact operator must be countably generated (\cite[Lemma 1.10]{Troitsky2020JMAA}), but $C_0(K)$ is not.
\end{proof}

Let us also prove that for the constructed example there is no non-standard frames, and we must start with the following useful lemma.

\begin{lem}\label{frame_restrict}
Let $K$ be a locally compact Hausdorff space, $A$ its closed subset. If there is a frame $\{x_j\}_{j\in J}$ (standard or not) in $C_0(K)$, then its restriction to $A$ $\{y_j\}_{j\in J }$ is a frame in $C_0(A)$ in the same sense.
\end{lem}

\begin{proof}
Since uniform convergence on a set implies uniform convergence on a subset, the proposition is obvious for standard frames. Let us prove it for non-standard.

Take $x\in C_0(A)$ and let $\f$ be a state on $C_0(A)$, %- i.e. positive linear functional,
i.e. a Radon measure $\mu$ on $A$. It can be extended to a measure on whole $K$ by zero outside $A$ --- that is, to the state $\f'$ on $C_0(K)$.
The function $x$ can be extended to the function $x'\in C_0(K)$ due to Lemma \ref{extrest}.
Since $\{x_j\}_{j\in J}$ is a frame, we have

$$
c_1\f'(\<x',x'\>)\le
\sum\limits_j \f'(\<x',x_j\>\<x_j,x'\>)\le
c_2\f'(\<x',x'\>).
$$

Since $\f'$ is a measure which is actually calculated on the restrictions of functions on $A$, for $x$ we have

$$
c_1\f(\<x,x\>)\le
\sum\limits_j \f(\<x,y_j\>\<y_j,x\>)\le
c_2\f(\<x,x\>),
$$
i.e. $\{y_j\}_{j\in J}$ is a non-standard frame in $C_0(K)$.
\end{proof}

\begin{teo}\label{ii_nocr_nofr}
For $K$ from the example \ref{ii_exnocr} there are no non-standard frames in $C_0(K)$.
\end{teo}

\begin{proof}
Suppose that there exists a non-standard frame $\{x_j\}_{j\in J}$ in $C_0(K)$.
Then its restriction $\{y_j\}_{j\in J}$ to $P\times\{\om_0\}$ is a non-standard frame in $C_0(P\times\{\om_0\})$, so $C_0(P)$ also has a frame, which cannot be, as we will see later (theorem \ref{iv_nofr}). A contradiction.
\end{proof}

\section{The properties of $\KIII$}

\begin{ex}\label{iii_good}
Due to \ref{iii_stfr} and \ref{iii_dop} as a ``good'' example when there exists a standard frame it suffices to take $K=\bigsqcup\limits_{\l\in\L}K_\l$ with uncountable $\L$, where all $K_\l$ are $\s$-compact. Indeed, any compact set in $K=\bigsqcup\limits_{\l\in\L}K_\l$ intersects only a finite number of $K_\l$, which means that a $\s$-compact set intersects only a countable number of them. Hence, the complement to any $\s$-compact subset contains some $K_\be$, which is an open set in $K$.
\end{ex}

Let us now introduce an example when the $\A$-compactness criterion is not satisfied, which implies that there is no standard frame; there is also no frames for it at all.

\begin{ex}\label{iii_bad}
Let $K=P_1\sqcup P_2$,
where $P_1\in\KIV$, $P_2\in\KI$, $\KII$ or $\KIII$. A $\s$-compact set in $K$ is a union of $\s$-compact sets from $P_1$ and $P_2$ respectively. The complement to a $\s$-compact set in $P_1$ is open, so the same is true for $K$. However, one can reach a point at infinity with a countable set of compact sets from $P_2$, so $K\in\KIII$.
\end{ex}

\begin{teo}\label{iii_bad_th}
Let $K$ be a space from example \ref{iii_bad}. Then the operator $F:C_0(K)\to C_0(K)$ of multiplication by the identity function on $P_1$ and by zero function on $P_2$ is not $C_0(K)$-compact, but the image of the unit ball with respect to this operator is $(C_b(K),C_0(K))$-totally bounded. Obviously, this operator is adjointable.
Also, $C_0(K)$ has no frames.
\end{teo}

\begin{proof}
The image of this operator is an uncountably generated module $C_0(P_1)$, so the operator cannot be $C_0(K)$-compact.

The image of the unit ball is the unit ball in $C_0(P_1)$, and since the restriction of the Radon measure to a measurable subset is the Radon measure (\cite[416R(b)]{Frem4}), the seminorm $\nu_{X,\F}$ on the image has the following form

$$
\nu_{X,\F}(x)^2=
\sup_k 
\sum_{i=k}^\infty |
\int\limits_K\overline{x(t)}\cdot x_i(t)d\mu_k(t)
|^2
=
\sup_k 
\sum_{i=k}^\infty |
\int\limits_{K'}\overline{x(t)}\cdot x_i(t)d\mu_k(t)
|^2
=$$
$$=
\sup_k 
\sum_{i=k}^\infty |
\int\limits_{P_1\sqcup P_2}\overline{x(t)}\cdot x_i(t)d\mu_k(t)
|^2
=
\sup_k 
\sum_{i=k}^\infty |
\int\limits_{P_1}\overline{x(t)}\cdot x_i(t)d\mu_k(t)
|^2,
$$
that is, the seminorm on the image is calculated as a seminorm on $C_0(P_1)$. Obviously, the restriction to $P_1$ of any $(C_b(K),C_0(K))$-admissible system is $(C_b(P_1),C_0(P_1))$-admissible.
Hence, the unit ball in $C_0(P_1)$
is $(C_b(K),C_0(K))$-totally bounded by reasons similar to \cite[Theorem 2.5]{Fuf2021faa} because the unit ball in $C_0(P_1)$ is $(C_b(P_1),C_0(P_1))$-totally bounded
(more specifically, since the elements of the $\e$-net, which are functions on $P_1$, can be extended to whole $K$).

If there exists a frame in $C_0(K)$, its restriction to $P_1$ would also be a frame, but as we will see later (theorem \ref{iv_nofr}), $C_0(P_1)$ has no frames, so there are no frames in $C_0(K)$ too.
\end{proof}

There is also an intermediate example of space: there is no standard frame, but a non-standard one exists.
But first let us prove the following another one useful lemma.

\begin{lem}
Let $P_1, P_2$ be locally compact Hausdorff spaces, and both $C_0(P_1)$ and $C_0(P_2)$ have frames $\{x_j\}_{j\in J}$ (with constants $ d_1,d_2$) and $\{y_i\}_{i\in I}$ (with constants $c_1,c_2$) respectively.
Then in $C_0(K)$, where $K=P_1\sqcup P_2$ there also exists a frame $\{w_g\}_{g\in G}=\{x_j\}_{j\in J}\cup \{y_i\}_{i\in I}$. If each of the original frames is standard, then $\{w_g\}_{g\in G}$ is standard too.
\end{lem}

\begin{proof}
Uniform convergence on a finite number of sets implies uniform convergence on their union. The same is true for the inequalities, so the case when both frames are standard is obvious. Let us prove for non-standard.

Let $\f$ be a state on $C_0(K)$, i.e. a measure on $K$. It can be represented as the sum of measures on $P_1$ and $P_2$ respectively, i.e. $\f=\f_1+\f_2$, where $\f_1, \f_2$ are states on $P_1, P_2$ respectively. It is also possible to represent in such a way the function $w\in C_0(K)$, $w=x+y$ and
$\<w,w\>=|w|^2=\<x,x\>+\<y,y\>$. Hence we get that

$$
\sum\limits_g \f(\<w,w_g\>\<w_g,w\>)=
\sum\limits_j \f_1(\<x,x_j\>\<x_j,x\>)+\sum\limits_i \f_2(\<y,y_i\>\<y_i,y\>),
$$
and hence

$$
\sum\limits_g \f(\<w,w_g\>\<w_g,w\>)\le
d_2\f_1(\<x,x\>)+c_2\f_2(\<y,y\>)=
$$
$$=
d_2\f(\<x,x\>)+c_2\f(\<y,y\>)\le
\max\{d_2, c_2\}(\f(\<x,x\>)+\f(\<y,y\>))=
\max\{d_2, c_2\}\f(\<w,w\>).
$$

Similarly, we have that
$$
\min\{d_1, c_1\}\f(\<w,w\>)
\le
\sum\limits_g \f(\<w,w_g\>\<w_g,w\>)\le
\max\{d_2, c_2\}\f(\<w,w\>),
$$
i.e. $\{w_g\}_{g\in G}$ is a frame in $C_0(K)$. Q.E.D.

\end{proof}

\begin{ex}\label{iii_bg}
Let $K=P_1\sqcup P_2$,
where $P_1$ is the space from example \ref{iii_good}, $P_2$ is the space from example \ref{ii_exfr}. Similar to the previous discussion, $K\in\KIII$.
\end{ex}

\begin{teo}
There is no standard frame in $C_0(K)$, but there exists a non-standard one.
\end{teo}

\begin{proof}
If $C_0(K)$ has a standard frame, its restriction to $P_1$ would also be a standard frame, but since $P_2\in\KII$, $C_0(P_2)$ has no standard frame (theorem \ref {ii_nostfr}), hence $C_0(K)$ also does not have.

Let us show that there exists a non-standard one. We know that in $C_0(P_1)$ there is a normalized standard frame $\{x_j\}_{j\in J}$, and in $C_0(P_2)$ there is a frame $\{y_i\}_{ i\in I}$ with constants $c_1=1$, $c_2=2$. 
By the previous lemma
their union $\{w_g\}_{g\in G}=\{x_j\}_{j\in J}\cup\{y_i\}_{i\in I}$ is a frame in $ C_0(K)$.
\end{proof}

 \section{Non-existence of non-standard frames in $\KIV$}

\begin{teo}\label{iv_nofr}
Let $K\in\KIV$. Then the $\dot{C_0}(K)$-module $C_0(K)$ has no frame.
\end{teo}

\begin{proof}

Assume that there exists a frame $\{x_j\}_{j\in J}$ in $C_0(K)$.
Take an arbitrary point $t_1\in K$.
There exists a function $g_1\in C_0(K)$ such that $g_1(t_1)=1$, $|g_1(t)|\le1$ on $K$.
There is a non-empty at most countable set $\{x_j\}_{j\in J_1}$ of elements of the frame such that $x_j(t_1)\ne0$ because by
taking $x=g_1$ and $\f={\delta_{t_1}}$ --- evaluation at the point $t_1$, due to lemma \ref{fr_cr} we have that
the series $\sum\limits_j \<g_1,x_j\>\<x_j,g_1\>(t)=\sum\limits_j |x_j(t)|^2$ converges at the point $t_1$. That is, if $j\in J\setminus J_1$, then $x_j(t_1)=0$.
For every $j\in J_1$ there is a compact set $K_{1,j}\subset K$ such that $x_j=0$ outside $K_{1,j}$. $J_1$ is at most a countable set, so there is a compact set $K_1$ such that $\bigcup\limits_{j\in J_1}K_{1,j}\subset K_1$. That is, $x_j=0$ outside $K_1$ for any $j\in J_1$.

Assume that we have already found points $t_1,\dots,t_n$, compact sets $K_1,\dots,K_n$ and index sets
%\linebreak
$J_1,\dots,J_n\subset J$ such that $t_i\notin \bigcup\limits_{l=1}^{i-1}K_l$ for $i=2,\dots,n$, $x_j (t_l)\ne0$ only for $j\in J_l$ (as a consequence, different sets $J_l$ do not intersect), $x_j=0$ outside $K_l$ for $j\in J_l$.

Take an arbitrary point $t_{n+1}\in K\setminus\bigcup\limits_{l=1}^{n}K_l$. As in the case when $n=1$, there exists a non-empty at most countable set $\{x_j\}_{j\in J_{n+1}}$ of elements of the frame such that $x_j(t_{n+1})\ne0$ for $j\in J_{n+1}$ (and hence $J_{n+1}$ does not intersect any $J_{l}$, $l=1,\dots,n$, since the functions $x_j$ for $j\in \bigcup\limits_{l=1}^{n}J_l$ vanishes outside $\bigcup\limits_{l=1}^{n}K_l$).
There also exists a compact set $K_{n+1}$ such that $x_j=0$ outside $K_{n+1}$ for any $j\in J_{n+1}$. By induction, we can continue this construction for any $n\in\N$.

The sequence $\{t_n\}_{n\in\N}$ is a $\s$-compact set, so there exists a compact set $K'$ containing this sequence.

Hence the sequence has a limit point $t_0\in K'$. Let us show that $x_j(t_0)=0$ for all $j\in J$, which will contradict the fact that $\{x_j\}_{j\in J}$ is a frame.

First let it be that $j\in J\setminus\bigcup\limits_{l=1}^{\infty}J_{l}$. Then $x_j(t_n)=0$ for all $n\in\N$. Hence,
$x_j(t_0)=0$, because otherwise if $x_j(t_0)=q\ne0$ then in any neighborhood of the point $t_0$ there is a point $t_n$ such that $|x_j(t_0)-x_j(t_n)|=|q|>0$ --- a contradiction with the continuity of $x_j$.

Let now $j\in\bigcup\limits_{l=1}^{\infty}J_{l}$, i.e. $j\in J_k$ for some $k\in\N$, and suppose that $x_j(t_0)\ne0$. Then $t_0\in K_k$ (because $x_j=0$ outside $K_k$). Hence, $x_j(t_{k+l})=0$ for all $l\in\N$ (because $t_{k+l}\notin K_k$) and $t_0$ is still a limit point for the sequence $\{t_n\}_{n=k+1}^\infty$, and then $x_j(t_0)=0$ similarly to the previous case.

Hence $\{x_j\}_{j\in J}$ is not a frame. Q.E.D.

\end{proof}


\begin{thebibliography}{10}

\bibitem{FrankLarson1999}
{\sc M. Frank, D.~R. Larson}.
\newblock A module frame concept for {H}ilbert {$C^\ast$}-modules.
\newblock {\em The functional and harmonic analysis of wavelets and frames
  ({S}an {A}ntonio, {TX}, 1999)}, volume 247 of {\em Contemp. Math.} (1999), 207--233. Amer. Math. Soc., Providence, RI.

\bibitem{FrankLarson2002}
{\sc M. Frank, D.~R. Larson}.
\newblock Frames in {H}ilbert {$C^\ast$}-modules and {$C^\ast$}-algebras.
\newblock {\em J. Operator Theory}, {\bf 48}:2 (2002), 273--314.

\bibitem{ArBak2017}
{\sc Lj. Arambasic, D. Bakic}.
\newblock Frames and outer frames for {H}ilbert ${C}^*$-Modules.
\newblock {\em Linear and multilinear algebra}, {\bf 65}:2 (2017), 381--431.

\bibitem{Bak2019}
{\sc D. Bakic}.
\newblock Weak frames in Hilbert $C^*$-modules with application in Gabor analysis.
\newblock {\em Banach J. Math. Anal.}, {\bf 13}:4 (2019), 1017--1075.

\bibitem{KeckicLazovic2018}
{\sc D.~J. Ke\v{c}ki\'{c}, Z.~Lazovi\'{c}}.
\newblock Compact and ``compact'' operators on standard {H}ilbert modules over
  {$W^*$}-algebras.
\newblock {\em Ann. Funct. Anal.}, {\bf 9}:2 (2018), 258--270.
\newblock (arXiv:1610.06956).

\bibitem{lazovic2018}
{\sc Z.~Lazovi\'c}.
\newblock Compact and ``compact'' operators on standard {H}ilbert modules over
  ${C}^*$-algebras.
\newblock {\em Adv. Oper. Theory}, {\bf 3}:4 (2018), 829--836.

\bibitem{Troitsky2020JMAA}
{\sc E.~V. Troitsky}.
\newblock Geometric essence of ``compact'' operators on {H}ilbert {C}*-modules.
\newblock {\em Journal of Mathematical Analysis and Applications}, {\bf 485}:2 (2020), 123842.


%\bibitem{Kasp}
%{\sc G.~G. Kasparov}.
%\newblock {H}ilbert {C*}-modules: theorems of {S}tinespring and {V}oiculescu.
%\newblock {\em J. Operator Theory}, {\bf 4} (1980), 133--150.
\bibitem{TroitFuf2020}
{\sc E.V. Troitsky, D.V. Fufaev}.
\newblock Compact Operators and Uniform Structures in Hilbert ${C}^*$-Modules.
\newblock {\em Funct. Anal. Its Appl.}, {\bf 54} (2020), 287--294.

%{\sc Е.В. Троицкий, Д.В. Фуфаев}.
%\newblock Компактные операторы и равномерные структуры в гильбертовых ${C}^*$-модулях.
%\newblock {\em Функц. анализ и его прил.}, {\bf 54}:4 (2020), 74--84.

\bibitem{Asadi2016}
{\sc M.B. Asadi}.
\newblock {Frames in right ideals of ${C}^*$-algebras}.
\newblock {\em Bull. Iranian Math. Soc.}, {\bf 42}:1 (2016), 61--67.

\bibitem{Fuf2021faa}
{\sc D.~V. Fufaev}.
\newblock A {H}ilbert {C}*-module with extremal properties.
\newblock 
{\em 
(Russian) Funktsional. Anal. i Prilozhen.}, {\bf 56}:1 (2022), 94--105;
{\em translation in
Funct. Anal. Appl. (to appear)  	arXiv:2107.03782}.

\bibitem{HLi2010}
{\sc H. Li}.
\newblock A {H}ilbert ${C}^*$-module admitting no frames.
\newblock {\em Bull. Lond. Math. Soc.}, {\bf 42}:3 (2010), 388--394.

\bibitem{AmAs2016}
{\sc  M. Amini, M.B. Asadi, G. Elliott, F. Khosravi}.
\newblock Frames in Hilbert $C^*$-modules and Morita equivalent $C^*$-algebras..
\newblock {\em Glasg. Math. J.}, {\bf 59}:1 (2017), 1--10.

\bibitem{FrAs2020}
{\sc  M.B. Asadi, M. Frank, Z. Hassanpour-Yakhdani}.
\newblock Frame-Less Hilbert $C^*$-modules II.
\newblock {\em Complex Anal. Oper. Theory}, {\bf 14}:32 (2020).

\bibitem{Lance}
{\sc E.~C. Lance}.
\newblock {\em Hilbert {C*}-modules - a toolkit for operator algebraists},
  volume 210 of {\em London Mathematical Society Lecture Note Series}.
\newblock Cambridge University Press, England, 1995.

\bibitem{MTBook}
{\sc V.M. Manuilov, E.V. Troitsky}.
\newblock {\em Hilbert ${C}^{*}$-Modules}.
\newblock American Mathematical Society, Providence, R.I., 2005.

\bibitem{ManuilovTroit2000JMS}
{\sc V.~M. Manuilov, E.~V. Troitsky}.
\newblock Hilbert ${C}^*$- and ${W}^*$-modules and their morphisms.
\newblock {\em Journal of Mathematical Sciences}, {\bf 98}:2 (2000), 137--201.

\bibitem{BakGul2002}
{\sc D. Bakic, B. Guljas}.
\newblock On a class of module maps of {H}ilbert ${C}^*$-modules.
\newblock {\em Math. Commun.}, {\bf 7}:2 (2002), 177--192.

\bibitem{BakGul2004}
{\sc D. Bakic, B. Guljas}.
\newblock Extensions of {H}ilbert ${C}^*$-modules, I.
\newblock {\em Houston J. Math.}, {\bf 30}:2 (2004), 537--558.

\bibitem{BourbakiTop2}
{\sc N. Bourbaki}.
\newblock {\em General topology. {C}hapters 5--10}.
\newblock Elements of Mathematics (Berlin). Springer-Verlag, Berlin, 1998.
\newblock Translated from the French, Reprint of the 1989 English translation.

\bibitem{KadRin1}
{\sc  R.V. Kadison, J.R. Ringrose}.
\newblock {\em Fundamentals of the Theory of Operator Algebras. Volume I: Elemantary theory}.
\newblock Am. Math. Soc., Providence, 1997.

\bibitem{BratRob}
{\sc  O. Bratteli, D.W. Robinson}.
\newblock {\em Operator Algebras and Quantum Statistical Mechanics 1},
\newblock Springer-Verlag, New York Heidelberg Berlin, 1979.

\bibitem{Kasp}
{\sc G.~G. Kasparov}.
\newblock {H}ilbert {C*}-modules: theorems of {S}tinespring and {V}oiculescu.
\newblock {\em J. Operator Theory} {\bf 4}:1(1980), 133--150.

\bibitem{Munkr}
{\sc  J.R. Munkres}.
\newblock {\em Topology},
\newblock Prentice Hall, Inc., Upper Saddle River, NJ, 2000.

\bibitem{Frem4}
{\sc  D.H. Fremlin}.
\newblock {\em Measure theory, vol. 4 Topological Measure Spaces}.
\newblock Torres Fremlin, Colchester, 2003.

\bibitem{Weaver}
{\sc  N. Weaver}.
\newblock {\em Measure Theory and Functional Analysis}.
\newblock World Scientific, 2014.

\bibitem{Yeh}
{\sc  J. Yeh}.
\newblock {\em Real Analysis: Theory of Measure and Integration}.
\newblock World Scientific, 2006.

\bibitem{Exampl}
{\sc  L.A. Steen, J.A. Seebach Jr.}.
\newblock {\em Counterexamples in Topology}.
\newblock Springer-Verlag, New York, 1978.

\bibitem{GilJer}
{\sc  L. Gillman, M. Jerison}.
\newblock {\em Rings of Continuous Functions}.
\newblock Springer-Verlag, New York, 1960.

\bibitem{KurMos}
{\sc  K. Kuratowski, A. Mostowski}.
\newblock {\em Set theory. With an introduction to descriptive set theory},
\newblock North-Holland Publishing Co., Amsterdam New York Oxford, 1976.

\bibitem{Engel}
{\sc R. Engelking}.
\newblock {\em General Topology},
\newblock PWN, Warsaw, 1977.

%\bibitem{KurMos}
%{\sc  К. Куратовский, А. Мостовский}.
%\newblock {\em Теория множеств}.
%\newblock Мир, Москва, 1970.

\bibitem{Averson}
{\sc  W. Arveson}.
\newblock {\em An Invitation to ${C}^*$-{A}lgebras}.
\newblock Springer-Verlag, New York, 1976.

\bibitem{Walker}
{\sc  R.C. Walker}.
\newblock {\em The Stone-\v{C}ech Compactification}.
\newblock Springer-Verlag, Berlin, Heidelberg, 1974.

\bibitem{FrM2020}
{\sc R. Eskandari, M. Frank, V.M. Manuilov, M.S. Moslehian}.
\newblock Extensions of the Lax–Milgram theorem to {H}ilbert ${C}^*$-modules
\newblock {\em  Positivity }, {\bf 24} (2020), 1169--1180.


\end{thebibliography}
\end{document}